\documentclass[12pt]{amsart}
\usepackage{amsmath,amsfonts,amsbsy,amsgen,amscd,mathrsfs,amssymb,amsthm}
\usepackage[usenames,dvipsnames]{xcolor}
\usepackage[colorlinks=true,citecolor=red,linkcolor=blue]{hyperref}
\usepackage{geometry}
\usepackage{setspace}
\allowdisplaybreaks[4]
\numberwithin{equation}{section}
\newtheorem{theorem}{Theorem}[section]

\newtheorem{lemma}[theorem]{Lemma}

\newtheorem{remark}[theorem]{Remark}

\title[Flow by Gauss Curvature to the Orlicz Chord Minkowski Problem]{Flow by Gauss Curvature to the Orlicz Chord Minkowski Problem}

\author{Xia Zhao and Peibiao Zhao}


\begin{document}
\begin{abstract}
The $L_p$ chord Minkowski problem based on chord measures and $L_p$ chord measures introduced firstly by
Lutwak, Xi, Yang and Zhang \cite{LE1} is a very important and meaningful geometric measure problem in the $L_p$ Brunn-Minkowski theory.  Xi, Yang, Zhang and Zhao \cite{XD} using variational methods gave a measure solution when $p > 1$ and $0<p<1$ in the symmetric case. Recently, Guo, Xi and Zhao \cite{GL} also obtained a measure solution for $0\leq p<1$ by similar methods without the symmetric assumption.

In the present paper, we investigate and confirm the orlicz chord Minkowski problem, which generalizes the $L_p$ chord Minkowski problem by replacing $p$ with  a fixed continuous function $\varphi:(0,\infty)\rightarrow(0,\infty)$, and achieve  the existence of smooth solutions to the orlicz chord Minkowski problem by using  methods of Gauss curvature flows.
\end{abstract}

\maketitle

\vskip 20pt
\section{Introduction and main results}

In the 1930s, Aleksandrov-Fenchel-Jessen first introduced the surface area measure of a convex body in Euclidean space,
which was a Borel measure on the unit sphere. Furthermore, the Minkowski problem was proposed, which plays a very important role in the study of convex bodies (compact, convex subsets with nonempty interiors) in Brunn-Minkowski theory. The classical Minkowski problem argues the existence,
uniqueness and regularity of a convex body whose surface area measure is equal to a pre-given Borel measure on the sphere. If the given measure has a positive continuous density, the Minkowski problem can be seen as the problem of prescribing the Gauss curvature in differential geometry.
The study of the Minkowski problem has a long history and has led to a series of influential works, such as Minkowski \cite{MH}, Lewy \cite{LH}, Nirenberg \cite{NL}, Pogorelov \cite{PA} and Cheng-Yau \cite{CS}, etc..

The dual Brunn-Minkowski theory \cite{LE} was introduced by Lutwak in the 1975, but, what acts as the dual counterpart of the geometric measures in the Brunn-Minkowski theory was not clear until the work of Huang, Lutwak, Yang and Zhang \cite{HY} in 2016. They discovered dual curvature measures and  posed the dual Minkowski problem which generated some interesting results, for examples \cite{BK, GR1, GR2, HY1, HY2, LR, ZY, ZY1}.
Based on the classical Minkowski problem, many Minkowski type problems have been introduced and extensively studied, naturally, the $L_p$ versions of Minkowski problem was proposed. The $L_p$ Minkowski problem
is the problem of prescribing $L_p$ surface area measure which was introduced by Lutwak \cite{LE0}; when $p=1$, that is the classical Minkowski problem. If the given measure has a density function, the $L_p$
Minkowski problem reduces to a Monge-Amp\`{e}re equation.
When $p>1$ and given measure is even, the $L_p$ Minkowski problem was solved by Lutwak \cite{LE0}, however, Chou and Wang \cite{CK} without evenness assumptions. Important contributions to critical cases of the $L_p$ Minkowski problem include
\cite{BK1, LE01, LE04, ZG, ZG1,ZG2} and more not listed references. However, two critical cases includes the logarithmic Minkowski problem ($p=0$) and the centro-affine Minkowski problem ($p=-n$) of the $L_p$ Minkowski problem still remain open. The solution of $L_p$ Minkowski problem was used to establish a $L_p$ affine Sobolev inequality which is stronger than the $L_p$ Sobolev inequality \cite{HC, LE02, LE03}.

In this paper, let $\mathcal{K}^n$ be the collection of convex bodies in Euclidean space $\mathbb{R}^n$. For the set of convex bodies containing the origin in their interiors in $\mathbb{R}^n$, we write $\mathcal{K}^n_o$ and donote the subset of $\mathcal{K}^n$ of convex bodies which are symmetric about the origin by $\mathcal{K}_e^n$. For $K\in \mathcal{K}^n$, the chord integral $I_q(K)$ of $K$ is defined as follows:
\begin{align*}
I_q(K)=\int_{\wp^n}|K\cap l|^qdl,\quad q\geq0,
\end{align*}
where $ |K \cap l|$ denotes the length of the chord $K\cap l$, and the integration is with respect to the (appropriately
normalized) Haar measure on the affine Grassmannian $\wp^n$ of lines in  $\mathcal{K}^n$. The volume $V(K)$ and surface area $S(K)$ are two important special forms of chord integral:
\begin{align*}I_1(K)=V(K),\qquad I_0(K)=\frac{\omega_{n-1}}{n\omega_n}S(K),\qquad I_{n+1}(K)=\frac{n+1}{\omega_n}V(K)^2,\end{align*}
where $\omega_n$ is the volume enclosed by the unit sphere $S^{n-1}$.

Recently, Lutwak, Xi, Yang and Zhang \cite{LE1} used the variational method to introduce the so-called $q$-th chord measure $F_q(K,\cdot)$  of $K$, that is the differential of $I_q(K)$. Roughly specking, for convex bodies $K$ and $L$ in $\mathbb{R}^n$, there has the following form,
\begin{align*}\frac{d}{dt}\bigg{|}_{t=0^+}I_q(K+tL)=\int_{S^{n-1}}h_L(v)dF_q(K,v),\qquad q\geq 0,\end{align*}
where $K+tL$ is the Minkowski combination between $K$ and $L$, $F_q(K,\cdot)$ is called the $q$-th chord measure of $K$ and $h_L$ is the support function of $L$. When $q=0,1$ in this formula, there are two special forms, which are the variational formulas of surface area and volume.
\begin{align*}F_0(K,\cdot)=\frac{(n-1)\omega_{n-1}}{n\omega_n}S_{n-2}(K,\cdot),\qquad F_1(K,\cdot)=S_{n-1}(K,\cdot),\end{align*}
where $S_{n-2}(K,\cdot)$ and $S_{n-1}(K,\cdot)$ are  respectively the $(n-2 )$-th order and $(n-1)$-th order area measure of $K$. The $(n-1)$-th order area measure is also called the classical surface area measure. Due to the emergence of the chord measure, naturally, the chord Minkowski problem has been proposed.

{\bf The  chord Minkowski problem :}~~{\it Suppose real $q\geq 0$. If $\mu$ is a finite Borel measure on $S^{n-1}$, what are the necessary and sufficient conditions for the existence of a convex body $K$ which satisfies the equation,
\begin{align*}
F_q(K,\cdot)=\mu?
\end{align*}}

This is a new Minkowski problem except $q=0,1$. The case of $q = 1$ is the classical
Minkowski problem for surface area measure, and the case of $q = 0$ is the unsolved
Christoffel-Minkowski problem for the $(n-2)$-th area measure. When $q > 0$, the chord
Minkowski problem was  completely solved in \cite{LE1}.

At the same time, the $L_p$ versions of the chord measure was also introduced in \cite{LE1}, it can be extended from the $L_p$ surface
area measure. If $L_p$ surface area measure is extended to a two-parameter family of geometric measures, called $L_p$
chord measures. The $(p,q)$-th chord measure, $F_{p,q}(K,\cdot)$ of $K\in\mathcal{K}_o^n$ is
defined as follows:
\begin{align}
\label{eq100}dF_{p,q}(K,\cdot)=h_K^{1-p}dF_q(K,\cdot),\quad p\in\mathbb{R},\quad q\geq 0,
\end{align}
where $h_K$ is the support function of $K$ and $F_q(K,\cdot)$ is the $q$-th chord measure of
$K$. When $q=1$, $F_{p,1}(K,\cdot)$ is the $L_p$ surface area measure.
When $q=0$, $F_{p,0}(K,\cdot)$ is the $L_p$ $(n-2)$-th area
measure. When $p=1$, $F_{1,q}(K,\cdot)$ is
just the $q$-th chord measure $F_q(K,\cdot)$. Surely, the $L_p$ chord Minkowski problem can be stated as follows.

{\bf The  $L_p$ chord Minkowski problem: }~~{\it
Let $\mu$ be a finite Borel measure on $S^{n-1}$, $p\in\mathbb{R}$ , and $q\geq 0$. What are the necessary and sufficient conditions for the existence of a convex body $K\in \mathcal{K}_o^n$ that solves the equation:
\begin{align}
\label{eq101}F_{p,q}(K,\cdot)=\mu?
\end{align}}

When $p=1$, it is the chord Minkowski problem, and $q=1$, this is the $L_p$ Minkowski problem.

When the given measure $\mu$ has an integrable positive density function $f$ on $S^{n-1}$, then the equation
(\ref{eq101}) turns into  a Monge-Amp\`{e}re-type partial differential equation:
 \begin{align*}\det(\nabla_{ij}h+h\delta_{ij})=\frac{h^{p-1}f}{\widetilde{V}_{q-1}([h],\nabla h)}, \qquad \text {on} \qquad  S^{n-1},
 \end{align*}
where $h$ is the unknown function on $S^{n-1}$, which is extended via homogeneity to
$\mathbb{R}^n$, while $\nabla h$ is the
Euclidean gradient of $h$ in $\mathbb{R}^n$, $\nabla_{ij}h$ is the spherical Hessian of $h$ with respect to an orthonormal frame on $S^{n-1}$, $\delta_{ij}$ is the Kronecker delta,
and ${\widetilde{V}_{q-1}([h],\nabla h)}$
is the $(q-1)$-th dual quermassintegral of the Wulff-shape
$[h]$ of $h$ with respect to  $\nabla h$ (see next section for the precise definition).

Correspondingly, the $L_p$ chord Minkowski problem was considered. Xi et.al. \cite{XD} obtained a measure
solution of the $L_p$ chord Minkowski problem when $p>1$ and the symmetric case of $0<p<1$ via the variational method. In very recent work, Guo et al.\cite{GL} improved results of \cite{XD}, they solved the $L_p$ chord Minkowski problem when $0 \leq p<1$ without symmetry assumptions by similar ways. Li \cite{LYY} treated the discrete $L_p$ chord Minkowski problem in the condition of $p<0$ and $q>0$, as for general Borel measure, Li also gave a proof but need $-n<p<0$ and $1<q<n+1$. In \cite{HJ}, Hu, Huang and Lu get regularity of the chord log-Minkowski problem of $p=0$ by flow methods, in addition, Hu, Huang, Lu and Wang \cite{HJ1} also using same flow with \cite{HJ} to obtain smooth origin-symmetric solution of $\{p>0, q>3\}\cup\{-n<p<0, 3<q<n+1\}$ for $L_p$ chord Minkowski problem. It should be pointed out the nonlocal term $\widetilde{V}_{q-1}$ with respect to boundary point of their flow in \cite{HJ} and \cite{HJ1}.

For problem (\ref{eq101}), and combining with (\ref{eq100}), we may consider the more generalized and normalized $L_\varphi$ chord Minkowski problem stated by the following form:

{\bf $L_\varphi$ chord Minkowski problem:}
Suppose $\varphi:(0,\infty)\rightarrow(0,\infty)$ is a fixed continuous function. If $\mu$ is a finite Borel measure on $S^{n-1}$ which is not concentrated on a great subsphere of $S^{n-1}$, what are the necessary and sufficient conditions on $\mu$ such that there exist a convex body $K\in\mathcal{K}^n_o$ in $\mathbb{R}^n$ and a positive constant $c$ such that $$d\mu=c\varphi(h_K)dF_{q}(K,\cdot)?$$

When the Radon-Nikodym derivative of $\mu$ for above problem with regard to the spherical measure on $S^{n-1}$ exists, namely $d\mu=\frac{1}{n}fdx$ for a nonnegative integrable function $f$, the above problem can be converted to the following new Monge-Amp\`{e}re equation on $S^{n-1}$:
\begin{align}\label{eq102} c\varphi(h_K)\widetilde{V}_{q-1}(K,\cdot)\det(\nabla_{ij}h+h\delta_{ij})=f.
\end{align}

When $\varphi(s)=s^{1-p}$ and $\varphi(s)=s$, the $L_\varphi$ chord Minkowski problem is reduced to the $L_p$ chord Minkowski problem and chord Minkowski problem, respectively.
\begin{remark}\label{rem11}
By analogy with the existing reference \cite{HC1},  $d\mu=c\varphi(h_K)dF_q(K,\cdot)$ that appears here is exactly the so-called Orlicz chord Minkowski problem, the corresponding Orlicz chord measure can also be induced by some geometric variation, which will be studied systematically in our next work. Here, we still call it shortly the $L_\varphi$ chord Minkowski problem.
\end{remark}
In the present paper, we will study the $L_\varphi$ chord Minkowski problem and give the existence of smooth solutions for (\ref{eq102}) by the method of Gauss curvature flows. The Gauss curvature flow was first introduced and studied by Firey \cite{FI} to model the shape change of worn stones. Since then, various Gauss curvature flows
have been extensively studied, see example \cite{AB, AB1, BR, BR1, CC, CL, CK0, GE, GU, HJ} and the references therein.

Let $\partial\Omega_0$ be a smooth, closed and strictly convex hypersurface in $\mathbb{R}^n$ enclosing the
origin in its interior. We consider the long-time existence and convergence of the following Gauss curvature
flow which is a family of convex hypersurfaces $\partial\Omega_t$  parameterized by  smooth maps $X(\cdot ,t):
S^{n-1}\times (0, \infty)\rightarrow \mathbb{R}^n$
satisfying the initial value problem
\begin{align}\label{eq103}
\left\{
\begin{array}{lc}
\frac{\partial X(x,t)}{\partial t}=-\theta(t)f(\nu)\frac{\langle X,\nu\rangle}{\widetilde{V}_{q-1}(X)\varphi(\langle X,\nu\rangle)}\mathcal{K}
(x,t)\nu+X(x,t),  \\
X(x,0)=X_0(x),\\
\end{array}
\right.
\end{align}
where $\mathcal{K}(x,t)$ is the Gauss curvature of hypersurface $\partial\Omega_t$,  $\nu=x$ is the
outer unit normal at $X(x,t)$, ${\langle X, \nu \rangle}$ stands for standard inner product and $\theta(t)$ is defined as follows
\begin{align*}\theta(t)=\frac{\int_{S^{n-1}}\widetilde{V}_{q-1}\rho^nd\xi}{\int_{S^{n-1}}\frac{hf}{\varphi}dx},\end{align*}
where $\rho$ and $h$ are the radial function and support function of the convex hypersurface $\partial \Omega_t$, respectively. Here, we need to point out that the flow (\ref{eq103}) is a normalised flow posed firstly in the present paper.

In order to discuss Gauss curvature flow (\ref{eq103}), we introduce a new functional for any $t\geq 0$,
\begin{align}\label{eq104}
\Phi(\Omega_t)=
\int_{S^{n-1}}f(x)\psi(h(x,t))dx,
\end{align}
where $\psi(s)=\int_{0}^{s}\frac{1}{\varphi(\tau)}d\tau$ exists for all $s>0$ and $\lim_{s\rightarrow \infty}\psi(s)=\infty $ and $h(\cdot,t)$ is the support function of $\Omega_t$.

Combining with problem (\ref{eq102}) and flow (\ref{eq103}), we establish the following theorem in this article.

\begin{theorem}\label{thm12}
 Suppose $\varphi:(0,\infty)\rightarrow (0,\infty)$ is a fixed continuous function, $\psi$ satisfy above definition, $\Omega_0$ be a smooth, closed and strictly convex
body in $\mathbb{R}^n$ enclosing the origin in its interior and $f$ be a positive
smooth function on $S^{n-1}$. then the normalised flow (\ref{eq103}) has a unique
smooth solution to the $\partial\Omega_t=X(S^{n-1},t)$ for $t\in(0,\infty)$ and a subsequence of
$\Omega_t$ converges to the smooth solution of (\ref{eq102}) with $\frac{1}{c}=\lim_{t_i\rightarrow\infty}\theta(t_i)>0$.
\end{theorem}

This paper is organized as follows. In Section \ref{sec2}, we collect some properties of convex bodies
and convex hypersurfaces. In Section \ref{sec3}, we give the parameterized form  of flow (\ref{eq103}) by
support function and discuss properties of two important functionals along the flow (\ref{eq103}).
In Section  \ref{sec4}, we give the priori estimates for the solution to the flow (\ref{eq103}). We obtain the
convergence of the normalised flow and complete the proof of Theorem \ref{thm12} in Section  \ref{sec5}.

\section{\bf Preliminaries}\label{sec2}

In this section, we give a brief review of some relevant notions about
convex bodies and recall some basic properties of convex hypersurfaces that readers may refer
to \cite{UR} and a celebrated book of Schneider \cite{SC}.

\subsection{Convex bodies} Let $\mathbb{R}^n$ be the $n$-dimensional Euclidean space
and $\partial\Omega$ be a smooth, closed and strictly convex hypersurface containing the origin in its interior. The support function of $\Omega$ is defined by
\begin{align*}h_\Omega(\xi)=h(\Omega,\xi)=\max\{\xi\cdot y:y\in\Omega\},\quad \forall\xi\in S^{n-1},\end{align*}
and the radial function of $\Omega$ with respect to $x\in\mathbb{R}$ is defined by
\begin{align*}\rho_{\Omega,x}(v)=\rho((\Omega,x),v)=\max\{c>0:cv+x\in\Omega\},\quad  v\in S^{n-1}.\end{align*}

Let $C^+(S^{n-1})$ denote the set of the positive continuous functions on $S^{n-1}$, if $h\in C^+(S^{n-1})$, the Wulff shape $[h]$ generated by $h$ is a convex body defined by
\begin{align*}[h]=\{x\in\mathbb{R}^n: x\cdot v\leq h(v), \forall v\in S^{n-1}\}.\end{align*}
Obviously, if $\Omega\in\mathcal{K}^n_o$, $[h_\Omega]=\Omega$.

If $\Omega\in \mathcal{K}^n$, for each $q\in\mathbb{R}$ and $z\in$ int$K$, the $q$-th dual
quermassintegral $\widetilde{V}_q(\Omega,z)$ of $\Omega$ with respect to $z$ is defined by
\begin{align}\label{eq201}\widetilde{V}_q(\Omega, z)=\frac{1}{n}\int_{S^{n-1}}\rho_{\Omega,z}^q(u)du.\end{align}

In addition, for $q > 0$, the chord integral can be written as the integral of dual
quermassintegrals with $z\in \Omega\setminus\{0\}$ in $\mathbb{R}^n$ (see\cite{LE1})
\begin{align}\label{eq202}
I_q(\Omega)=\frac{q}{\omega_n}\int_\Omega\widetilde{V}_{q-1}(\Omega,z)dz.\end{align}

For a compact convex subset $K\in \mathcal{K}^n$ and $v\in S^{n-1}$, the intersection of a
supporting hyperplane with $K$, $H(K,v)$ at $v$ is given by
\begin{align*}H(K,v)=\{x\in K:x\cdot v=h_K(v)\}.\end{align*}
A boundary point of $K$ which only has one supporting hyperplane is called a regular point, otherwise, it is a singular point. The set of singular points is denoted as $\sigma K$, it is
well known that $\sigma K$ has spherical Lebesgue measure 0.

For $x\in\partial K\setminus \sigma K$, its Gauss map $g_K:x\in\partial K\setminus \sigma K\rightarrow S^{n-1}$ is represented by
\begin{align*}g_K(x)=\{v\in S^{n-1}:x\cdot v=h_K(v)\}.\end{align*}
Correspondingly, for a Borel set $\eta\subset S^{n-1}$, its inverse Gauss map is denoted by
$g_K^{-1}$,
\begin{align*}g_K^{-1}(\eta)=\{x\in\partial K:g_K(x)\in\eta\}.\end{align*}
For the Borel set $\eta\subset S^{n-1}$, its surface area measure is defined as
\begin{align*}S_K(\eta)=\mathcal{H}^n(g_K^{-1}(\eta)),\end{align*}
where $\mathcal{H}^n$ is $n$-dimensional Hausdorff measure.

\subsection{Convex hypersurface }~~Suppose that $\Omega$ is parameterized by the inverse
Gauss map $X:S^{n-1}\rightarrow \Omega$, namely $X(x)=g_\Omega^{-1}(x)$. The support function $h$ of $\Omega$ can be computed by
\begin{align}\label{eq203}h(x)={\langle x, X(x) \rangle}, \ \ x\in S^{n-1},\end{align}
where $x$ is the outer normal of $\Omega$ at $X(x)$. Denote  $e_{ij}$ and $\nabla$ by the standard metric and gradient on the sphere $S^{n-1}$, respectively.
Differentiating (\ref{eq203}), there has
\begin{align*}\nabla_ih=\langle\nabla_ix,X(x)\rangle+\langle x,\nabla_iX(x)\rangle,\end{align*}
since $\nabla_iX(x)$ is tangent to $\Omega$ at $X(x)$, thus
\begin{align*}\nabla_ih=\langle\nabla_ix,X(x)\rangle.\end{align*}
It follows that
\begin{align*} X(x)=\nabla h+hx.\end{align*}

By differentiating (\ref{eq203}) twice, the second fundamental form $A_{ij}$ of $\Omega$ can be computed in terms of the support function,
\begin{align}\label{eq204}A_{ij} = \nabla_{ij}h + he_{ij},\end{align}
where $\nabla_{ij}=\nabla_i\nabla_j$ denotes the second order covariant derivative with respect to $e_{ij}$. The induced metric matrix $g_{ij}$ of $\Omega$ can be derived by Weingarten's formula,
\begin{align}\label{eq205}e_{ij}=\langle\nabla_ix, \nabla_jx\rangle= A_{ik}A_{lj}g^{kl}.\end{align}
The principal radii of curvature are the eigenvalues of the matrix $b_{ij} = A^{ik}g_{jk}$.
When considering a smooth local orthonormal frame on $S^{n-1}$, by virtue of (\ref{eq204})
and (\ref{eq205}), there has
\begin{align}\label{eq206}b_{ij} = A_{ij} = \nabla_{ij}h + h\delta_{ij}.\end{align}
The Gauss curvature of $X(x)\in\Omega$ is given by
\begin{align}\label{eq207}\mathcal{K}(x) = (\det (\nabla_{ij}h + h\delta_{ij} ))^{-1}.\end{align}

\section{\bf Gauss curvature flow and its associated functionals}\label{sec3}
In this section, we shall introduce the geometric flow and its associated functionals for solving the $L_\varphi$ chord Minkowski problem. For convenience, the curvature flow is restated
here. Let $\partial\Omega_0$ be a smooth, closed and strictly convex hypersurface in $\mathbb{R}^n$ enclosing
the origin in its interior, $f$ be a positive
smooth function on $S^{n-1}$. We consider the following Gauss curvature flow
\begin{align}\label{eq301}
\left\{
\begin{array}{lc}
\frac{\partial X(x,t)}{\partial t}=-\theta(t)f(\nu)\frac{\langle X,\nu\rangle}{\widetilde{V}_{q-1}(X)\varphi(\langle X,\nu\rangle)}\mathcal{K}
(x,t)\nu+X(x,t),  \\
X(x,0)=X_0(x),\\
\end{array}
\right.\end{align}
where $\mathcal{K}(x,t)$ is the Gauss curvature of the hypersurface $\partial\Omega_t$ at $X(\cdot,t)$, $\nu=x$ is the unit outer
normal vector of $\partial\Omega_t$ at $X(\cdot,t)$ and for any $u\in S^{n-1}$, $\theta (t)$ is given by
\begin{align}\label{eq302}\theta(t)=\frac{\int_{S^{n-1}}\widetilde{V}_{q-1}\rho^nd\xi}{\int_{S^{n-1}}\frac{hf}{\varphi}dx}.\end{align}

According to (\ref{eq203}), taking the scalar product of both sides of the equation and of the initial condition in
(\ref{eq301}) by $\nu$, by means of the definition of support function, we describe the flow equation
associated with the support function as
\begin{align}\label{eq303}\left\{
\begin{array}{lc}
\frac{\partial h(x,t)}{\partial t}=-\theta(t)f(x)\frac{h(x,t)}{\widetilde{V}_{q-1}([h])\varphi(h)}\mathcal{K}
(x,t)+h(x,t),  \\
h(x,0)=h_0(x),\\
\end{array}
\right.\end{align}

Next, we investigate the characteristics of two essential geometric functionals along the
flow (\ref{eq301}), before that, let's list a fact firstly (see e.g. \cite{LY}).
\begin{align}\label{eq304}\frac{1}{\rho(\xi,t)}\frac{\partial\rho(\xi,t)}{\partial t}=\frac{1}{h(x,t)}\frac{\partial
h(x,t)}{\partial t}.\end{align}

\begin{lemma}\label{lem31} Let $X(\cdot,t)$ be a smooth solution to the flow (\ref{eq301}) and for any $t\geq0$, $\partial\Omega_t=X(S^{n-1},t)$ be a smooth, closed strictly
convex hypersurface in $\mathbb{R}^n$. Suppose that the origin lies in the interior of the convex body $\Omega_t$ enclosed by $\partial\Omega_t$. Then for
$q>0$, $I_q(\Omega_t)$ is unchanged, i.e.
\begin{align*}I_q(\Omega_t)=I_q(\Omega_0).\end{align*}
\end{lemma}
\begin{proof}Let $h(\cdot,t)$ and $\rho(\cdot,t)$ be the support function and radial function of $\Omega_t$, respectively. Applying polar coordinates, (\ref{eq202}), (\ref{eq302}),(\ref{eq303}),(\ref{eq304}) and $\rho^n\mathcal{K}d\xi=hdx$, we have
\begin{align*}
\frac{d}{dt}I_q(\Omega_t)=&\frac{d}{dt}\bigg(\frac{q}{\omega_n}\int_{\Omega_t}\widetilde{V}_{q-1}(\Omega_t,z)dz\bigg)\\
=&\frac{q}{\omega_n}\frac{d}{dt}\bigg(\int_{S^{n-1}}\int_0^{\rho(\xi,t)}\widetilde{V}_{q-1}(\Omega_t,\rho(\xi,t))d\rho d\xi\bigg)\\
=&\frac{q}{\omega_n}\int_{S^{n-1}}\widetilde{V}_{q-1}(\Omega_t,\rho(\xi,t))\rho(\xi,t)^{n-1}\frac{\partial\rho (\xi,t)}{\partial t} d\xi\\
=&\frac{q}{\omega_n}\int_{S^{n-1}}\widetilde{V}_{q-1}
(\Omega_t,\rho(\xi,t))\frac{\rho(\xi,t)^{n}\mathcal{K}}{\rho(\xi,t)\mathcal{K}}\frac{\partial\rho (\xi,t)}{\partial t} d\xi\\
=&\frac{q}{\omega_n}\int_{S^{n-1}}\widetilde{V}_{q-1}
(\Omega_t,\rho(\xi,t))\frac{h(x,t)}{\rho(\xi,t)\mathcal{K}}\frac{\partial\rho (\xi,t)}{\partial t}dx\\
=&\frac{q}{\omega_n}\int_{S^{n-1}}\widetilde{V}_{q-1}
(\Omega_t,\rho(\xi,t))\frac{1}{\mathcal{K}}\frac{\partial h (x,t)}{\partial t}dx\\
=&\frac{q}{\omega_n}\bigg[\int_{S^{n-1}}\widetilde{V}_{q-1}
\frac{1}{\mathcal{K}}\bigg(-\theta(t)(\varphi\widetilde{V}_{q-1})^{-1}\mathcal{K}hf\bigg)dx+\int_{S^{n-1}}\frac{h\widetilde{V}_{q-1}}
{\mathcal{K}}dx\bigg]\\
=&\frac{q}{\omega_n}\bigg[-\frac{\int_{S^{n-1}}\widetilde{V}_{q-1}\rho^nd\xi}{\int_{S^{n-1}}\frac{hf}{\varphi}dx}\int_{S^{n-1}}\widetilde{V}_{q-1}
\frac{f\mathcal{K}}{\mathcal{K}}\frac{h}{\widetilde{V}_{q-1}\varphi}dx+\int_{S^{n-1}}\frac{h\widetilde{V}_{q-1}}
{\mathcal{K}}dx\bigg]\\
=&\frac{q}{\omega_n}\bigg[-\frac{\int_{S^{n-1}}\widetilde{V}_{q-1}\rho^nd\xi}{\int_{S^{n-1}}\frac{hf}{\varphi}dx}\int_{S^{n-1}}
\frac{fh}{\varphi}dx+\int_{S^{n-1}}\frac{h\widetilde{V}_{q-1}}
{\mathcal{K}}dx\bigg]\\
=&0.
\end{align*}
This ends the proof of Lemma \ref{lem31}.\end{proof}

\begin{lemma}\label{lem32}
The functional (\ref{eq104}) is non-increasing along the flow (\ref{eq301}). Namely, $\frac{d}{dt}\Phi(\Omega_t)\leq0$, and the equality holds if and only if $\Omega_t$ satisfy (\ref{eq102}).
\end{lemma}
\begin{proof}
By (\ref{eq104}), (\ref{eq302}), (\ref{eq303}), $\rho^n\mathcal{K}d\xi=hdx$ and H\"{o}lder inequality, we obtain following result,
\begin{align*}
\frac{d}{dt}\Phi(\Omega_t)=&\int_{\mathbb{S}^{n-1}}f(x) \psi'(h(x,t))\frac{\partial h}{\partial t}dx\\
=&\int_{S^{n-1}}\bigg(-\theta(t)\frac{f(x)h}{\widetilde{V}_{q-1}\varphi(h)}\mathcal{K}
+h\bigg)\frac{f(x)}{\varphi(h)}dx\\
=&-\theta(t)\int_{S^{n-1}}\frac{f^{2}(x)h}{\widetilde{V}_{q-1}\varphi^2(h)}\mathcal{K}dx+\int_{S^{n-1}}\frac{hf(x)}{\varphi(h)}dx\\
=&-\frac{\int_{S^{n-1}}\widetilde{V}_{q-1}\frac{h}
{\mathcal{K}}dx}{\int_{S^{n-1}}\frac{hf}{\varphi}dx}\int_{S^{n-1}}\frac{f^{2}(x)h}{\widetilde{V}_{q-1}\varphi^2(h)}\mathcal{K}dx+\int_{S^{n-1}}\frac{hf(x)}{\varphi(h)}dx\\
=&\bigg(\int_{S^{n-1}}\frac{hf(x)}{\varphi(h)}dx\bigg)^{-1}\bigg[-\int_{S^{n-1}}\bigg(\sqrt{\widetilde{V}_{q-1}\frac{h}
{\mathcal{K}}}\bigg)^2dx\int_{S^{n-1}}\bigg(\sqrt{\frac{f^{2}(x)h}{\widetilde{V}_{q-1}\varphi^2(h)}\mathcal{K}}\bigg)^2dx\\
&+\bigg(\int_{S^{n-1}}\frac{hf(x)}{\varphi(h)}dx\bigg)^2\bigg]\\
\leq&\bigg(\int_{S^{n-1}}\frac{hf(x)}{\varphi(h)}dx\bigg)^{-1}\bigg[-\bigg(\int_{S^{n-1}}\sqrt{\widetilde{V}_{q-1}\frac{h}
{\mathcal{K}}}\sqrt{\frac{f^{2}(x)h}{\widetilde{V}_{q-1}\varphi^2(h)}\mathcal{K}}dx\bigg)^2\\
&+\bigg(\int_{S^{n-1}}\frac{hf(x)}{\varphi(h)}dx\bigg)^2\bigg]=0.
\end{align*}

By the equality condition of H\"{o}lder inequality, we know that the above equality holds if and only if $f=c\varphi\frac{1}{\mathcal{K}}\widetilde{V}_{q-1}$, i.e.,
$$ c\varphi\widetilde{V}_{q-1}\det(\nabla_{ij}h+h\delta_{ij})=f.$$
Namely, $\Omega_t$ satisfies (\ref{eq102}) with $\frac{1}{c}=\theta(t)$.
\end{proof}

\section{\bf Priori estimates}\label{sec4}

In this section, we establish the $C^0, C^1$ and $C^2$ estimates for the solution to equation (\ref{eq303}). In the following of this paper, we always assume that $\partial\Omega_0$ is a smooth, closed and strictly convex hypersurface in $\mathbb{R}^n$, $h:S^{n-1}\times [0,T)\rightarrow \mathbb{R}$ is a smooth solution to the evolution equation (\ref{eq303}) with the initial $h(\cdot,0)$ the support function of $\Omega_0$. Here $T$ is the maximal time for which the smooth solution exists.

\subsection{$C^0, C^1$  bounds}

In order to complete the priori estimates, we first introduce the following Lemma for convex bodies, one can see
\cite[Lemma 2.6]{CH} for the details.
\begin{lemma}\label{lem41}
Let $\Omega\in\mathcal{K}^n_o$, $h$ and $\rho$ be support
function and radial function of $\Omega$, and $x_{\max}$ and $\xi_{\min}$ be two points such that
$h(x_{\max})=\max_{S^{n-1}}h$ and $\rho(\xi_{\min})=\min_{S^{n-1}}\rho$. Then
\begin{align*}
\max_{S^{n-1}}h=&\max_{S^{n-1}}\rho \quad \text{and} \quad \min_{S^{n-1}}h=\min_{S^{n-1}}\rho,\end{align*}
\begin{align*}h(x)\geq& x\cdot x_{\max}h(x_{\max}),\quad \forall x\in S^{n-1},\end{align*}
\begin{align*}\rho(\xi)\xi\cdot\xi_{\min}\geq&\rho(\xi_{\min}),\quad \forall x\in S^{n-1}.\end{align*}
\end{lemma}

\begin{lemma}\label{lem42}
Let $\partial\Omega_t$ be a smooth solution to the flow (\ref{eq301}) in $\mathbb{R}^n$ enclosing the origin in its interior, $f$ be a
positive smooth function on $S^{n-1}$ and $\varphi:(0,\infty)\rightarrow (0,\infty)$ is a fixed continuous function. Then there is a positive constant $C$ independent of
$t$ such that
\begin{align}\label{eq401}
\frac{1}{C}\leq h(x,t)\leq C, \ \ \forall(x,t)\in S^{n-1}\times[0,T),
\end{align}
\begin{align}\label{eq402}
\frac{1}{C}\leq \rho(u,t)\leq C, \ \ \forall(u,t)\in S^{n-1}\times[0,T).
\end{align}
Here, $h(x,t)$ and $\rho(u,t)$ are the support function and radial function of $\Omega_t$, respectively.
\end{lemma}
\begin{proof}
We only prove (\ref{eq401}), (\ref{eq402}) can be obtained by Lemma \ref{lem41} and (\ref{eq401}).

Firstly, we prove the upper bound. From Lemma \ref{lem41}, at any fixed time $t$, we have $h(x,t)\geq(x\cdot x^t_{\max})h_{\max}(t)$, $\forall x\in S^{n-1}$, where $x^t_{\max}$ is the
point such that $h(x^t_{\max})=\max_{S^{n-1}}h(\cdot,t)$ and $h_{\max}(t)=\max_{S^{n-1}}h(\cdot,t)$.

By Lemma \ref{lem32} and recall definition of $\psi$, we know that $\psi$ is strictly increasing. Thus
\begin{align*}
\Phi(\Omega_0)\geq&\Phi(\Omega_t)=\int_{S^{n-1}}f(x)\psi(h(x,t))dx\\
\geq&\int_{\{x\in{S^{n-1}}:x\cdot x^t_{\max}\geq\frac{1}{2}\}}f(x)\psi(h(x,t))dx\\
\geq&\int_{\{x\in{S^{n-1}}:x\cdot x^t_{\max}\geq\frac{1}{2}\}}f(x)\psi(\frac{1}{2}h_{max}(t))dx\\
\geq&C\psi(\frac{1}{2}h_{max}(t)),
\end{align*}
which implies that $\psi(\frac{1}{2}h_{max}(t))$ is uniformly bounded from above. Since $\psi$ is strictly increasing, we can know that $h(\cdot,t)$ has  uniformly bounded from above. Here $C$ is a positive constant depending only on $\max_{S^{n-1}}h(x,0), \min_{S^{n-1}}h(x,0)$ and $\max_{S^{n-1}}f(x), \min_{S^{n-1}}f(x)$.

In the following, we use contradiction to prove the lower bound of $h(x,t)$.   
Suppose that ${t_k}\subset[0,T)$ be a
sequence such that $h(x,t_k)$ is not uniformly bounded away from $0$, with the aid of the upper bound, using Blaschke selection theorem \cite{SC}, there is a sequence in ${\Omega_{t_k}}$, which is still denoted by ${\Omega_{t_k}}$ converges to a convex body
contained in a lower-dimensional subspace. This can lead to $\rho(u,t_k)\rightarrow 0$, as
$k\rightarrow\infty$ almost everywhere with respect to the spherical Lebesgue measure.
Combining with bounded convergence theorem, we can derive
\begin{align*}I_q(\Omega_{t_k})=&\frac{q}{\omega_n}\int_{\Omega_{t_k}}\widetilde{V}_{q-1}(\Omega_{t_k},z)dz\\
=&\frac{q}{\omega_n}\int_{S^{n-1}}du\int_0^{\rho(u,t_k)}\widetilde{V}_{q-1}(\Omega_{t_k},\rho(u,t_k))d\rho\rightarrow 0
\end{align*}
as $k\rightarrow\infty$ when $q>0$, which is a contradiction to Lemma \ref{lem31} with $I_q(\Omega_{t_k})=I_q(\Omega_0)$. Therefore, we complete proof.
\end{proof}

\begin{lemma}\label{lem43}Let $\Omega_t$, $f$ and $\varphi$ be as Lemma \ref{lem42}, then there is a positive constant $C$ independent of $t$ such that
\begin{align}\label{eq403}|\nabla h(x,t)|\leq C,\quad\forall(x,t)\in S^{n-1}\times [0,T),\end{align}
and
\begin{align}\label{eq404}|\nabla \rho(u,t)|\leq C,\quad \forall(u,t)\in S^{n-1}\times [0,T).\end{align}
\end{lemma}

\begin{proof}
The desired results immediately follows from Lemma \ref{lem42} and the identities (see e.g.\cite{LR}) as follows
\begin{align*}
h=\frac{\rho^2}{\sqrt{\rho^2+|\nabla\rho|^2}},\qquad\rho^2=h^2+|\nabla h|^2.\end{align*}
\end{proof}

\begin{lemma}\label{lem44}Under the conditions of Lemma \ref{lem42}, there always exists a positive constant $C$ independent of $t$, such that
\begin{align*}\frac{1}{C}\leq\theta(t)\leq C,\quad t\in [0,T).\end{align*}
\end{lemma}

\begin{proof}By the definition of $\theta(t)$,
$$\theta(t)=\frac{\int_{S^{n-1}}\widetilde{V}_{q-1}\rho^nd\xi}{\int_{S^{n-1}}\frac{hf}{\varphi}dx},$$
the conclusion of this result is directly obtained from the Lemma \ref{lem42} and the definition of $\widetilde{V}_{q-1}$ (see formula (\ref{eq201})).\end{proof}

\subsection{$C^2$ bounds}

In this subsection, we establish the upper and lower bounds of principal curvature. This will shows
that the equation (\ref{eq303}) is uniformly parabolic. The technique used in this proof was first introduced by Tso \cite{TK} to derive the upper bound of the Gauss curvature, see also the proof of Lemma 5.1 in \cite{HJ0} etc.. 

By Lemma \ref{lem42} and Lemma \ref{lem43}, if $h$ is a smooth solution of (\ref{eq303}) on $S^{n-1}\times [0,T)$
and $f, \varphi$ satisfying assumption of Lemma \ref{lem42}, then along the flow for $[0,T), \nabla h+hx$, and $h$ are smooth functions whose
ranges are within some bounded domain $\Omega_{[0,T)}$ and bounded interval $I_{[0,T)}$, respectively. Here $\Omega_{[0,T)}$ and $I_{[0,T)}$ depend only
on the upper and lower bounds of $h$ on $[0,T)$.

\begin{lemma}\label{lem45}Let $\Omega_t$, $f$  and $\varphi$ be as Lemma \ref{lem42} and $q>2$, there is a positive constant $C$ depending on
 $\|f\|_{C^0(S^{n-1})}, \|\varphi\|_{C^1(I_{[0,T)})} ,\|\varphi\|_{C^2(I_{[0,T)})}$, $\|h\|_{C^1(S^{n-1}\times [0,T)}$ and $\|\theta\|_{C^0(S^{n-1}\times [0,T)}$,
where $\|s\|$ represent $\max s$  and $\min s$, such that the principal curvatures $\kappa_i$ of $\Omega_t$, $i=1,\cdots, n-1$, are
bounded from above and below, satisfying
\begin{align*}\frac{1}{C}\leq \kappa_i(x,t)\leq C, \quad\forall (x,t)\in S^{n-1}\times [0,T).\end{align*}
\end{lemma}

\begin{proof} The proof is divided into two parts: in the first part, we derive an upper bound for the Gauss curvature $\mathcal{K}(x,t)$; in the second part, we derive an upper bound for the principal radii $b_{ij}=h_{ij}+h\delta_{ij}$.

Step 1: Prove $\mathcal{K}\leq C$.

Firstly, we construct the following auxiliary function,
\begin{align*}Q(x,t)=\frac{\theta(t)(\varphi\widetilde{V}_{q-1})^{-1}\mathcal{K}f(x)h-h}{h-\varepsilon_0}
\equiv\frac{-h_t}{h-\varepsilon_0},\end{align*}
where
\begin{align*}\varepsilon_0=\frac{1}{2}\min_{S^{n-1}\times[0, T)}h(x,t)>0.\end{align*}

For any fixed $t\in[0,T)$, we assume that $\max_{S^{n-1}}Q(x,t)$ is attained at $x_0$. Then at $x_0$, we have
\begin{align}\label{eq405}
0=\nabla_iQ=\frac{-h_{ti}}{h-\varepsilon_0}+\frac{h_th_i}{(h-\varepsilon_0)^2},
\end{align}
and from (\ref{eq405}), at $x_0$, we also get
\begin{align}\label{eq406}
\nonumber 0\geq\nabla_{ii}Q=&\frac{-h_{tii}}{h-\varepsilon_0}+\frac{h_{ti}h_{i}}{(h-\varepsilon_0)^2}
+\frac{h_{ti}h_{i}+h_th_{ii}}{(h-\varepsilon_0)^2}-\frac{h_th_i(2(h-\varepsilon_0)h_i)}{(h-\varepsilon_0)^4}\\
\nonumber=&\frac{-h_{tii}}{h-\varepsilon_0}+\frac{h_{ti}h_{i}+h_{ti}h_i+h_th_{ii}}
{(h-\varepsilon_0)^2}
-\frac{2h_th_ih_i}{(h-\varepsilon_0)^3}\\
\nonumber=&\frac{-h_{tii}}{h-\varepsilon_0}+\frac{h_th_{ii}}
{(h-\varepsilon_0)^2}
+\frac{(h_{ti}h_i+h_{ti}h_i)(h-\varepsilon_0)-2h_th_ih_i}{(h-\varepsilon_0)^3}\\
\nonumber=&\frac{-h_{tii}}{h-\varepsilon_0}+\frac{h_th_{ii}}
{(h-\varepsilon_0)^2}
+\frac{h_th_ih_i+h_th_ih_i-2h_th_ih_i}{(h-\varepsilon_0)^3}\\
=&\frac{-h_{tii}}{h-\varepsilon_0}+\frac{h_th_{ii}}
{(h-\varepsilon_0)^2}.
\end{align}
From (\ref{eq406}), we obtain
$$-h_{tii}\leq\frac{-h_th_{ii}}{h-\varepsilon_0}.$$
Hence,
\begin{align}\label{eq407}\nonumber-h_{tii}-h_t\delta_{ii}\leq&\frac{-h_th_{ii}}{h-\varepsilon_0}-h_t\delta_{ii}=\frac{-h_t}
{h-\varepsilon_0}(h_{ii}+(h-\varepsilon_0)\delta_{ii})\\
=&Q(h_{ii}+h\delta_{ii}-\epsilon_0\delta_{ii})
=Q(b_{ii}-\varepsilon_0\delta_{ii}).\end{align}
At $x_0$, we also have
\begin{align*}\frac{\partial}{\partial t}Q=&\frac{-h_{tt}}{h-\epsilon_0}+\frac{h_t^2}{(h-\epsilon_0)^2}\\
=&\frac{f}{h-\epsilon_0}\bigg[\frac{\partial(\theta(t)(\varphi\widetilde{V}_{q-1})^{-1}h)}{\partial t}\mathcal{K}+\theta(t)(\varphi\widetilde{V}_{q-1})^{-1}h\frac{\partial(\det(\nabla^2h+hI))^{-1}}{\partial t}\bigg]+Q+Q^2,\end{align*}
where
\begin{align*}
\frac{\partial}{\partial t}((\varphi\widetilde{V}_{q-1})^{-1}h)=&-\varphi^{-2}\varphi^{\prime}\frac{\partial h}{\partial t}(\widetilde{V}_{q-1})^{-1}h-\frac{q-1}{n}(\widetilde{V}_{q-1})^{-2}\varphi^{-1} h\int_{S^{n-1}}\rho^{q-2}\frac{\partial\rho}{\partial t}du\\
&+(\varphi\widetilde{V}_{q-1})^{-1}\frac{\partial h}{\partial t},
\end{align*}
where $\varphi^{\prime}$ denotes $\frac{\partial\varphi(s)}{\partial s}$.

From $\rho^2=h^2+|\nabla h|^2$, (\ref{eq407}) and $x_0$ is a maximum of $Q$, we get
\begin{align*}\frac{\partial \rho}{\partial t}=\rho^{-1}(hh_t+\sum h_kh_{kt})=\rho^{-1}Q(\varepsilon_0h-\rho^2)\leq\rho^{-1}Q(x_0,t)(\varepsilon_0h-\rho^2),\end{align*}
thus, for $q>1$, we have
\begin{align*}
\frac{\partial}{\partial t}((\varphi\widetilde{V}_{q-1})^{-1}h)\leq&\frac{(q-1)Q(x_0,t)(\rho^2-\varepsilon_0h)}{n\rho(\widetilde{V}_{q-1})^2\varphi}\varphi^{-1}h\int_{S^{n-1}}\rho^{q-2}du\\
&+\varphi^{-2}\varphi^{\prime}Q(x_0,t)(h-\varepsilon_0)(\widetilde{V}_{q-1})^{-1}h+(\widetilde{V}_{q-1}\varphi)^{-1}(-Q(h-\varepsilon_0))\\
\leq& C_1Q.
\end{align*}
And from (\ref{eq302}), we know that
\begin{align*}\frac{\partial}{\partial t}(\theta(t))=&\frac{\partial}{\partial t}\bigg(\frac{\int_{S^{n-1}}
\widetilde{V}_{q-1}\rho^nd\xi}{\int_{S^{n-1}}hf/\varphi dx}\bigg)\\
=&\frac{\frac{\partial}{\partial t}
\bigg(\int_{S^{n-1}}\widetilde{V}_{q-1}\rho^nd\xi\bigg)}{\int_{S^{n-1}}hf/\varphi dx}-\frac{\frac{\partial}{\partial t}\bigg(\int_{S^{n-1}}hf/\varphi dx\bigg)\bigg(\int_{S^{n-1}}\widetilde{V}_{q-1}\rho^nd\xi\bigg)}{\bigg(\int_{S^{n-1}}hf/\varphi dx\bigg)^2},\end{align*}
for $q>1$, one can obtain
\begin{align*}&\frac{\partial}{\partial t}\bigg(\int_{S^{n-1}}\widetilde{V}_{q-1}\rho^nd\xi\bigg)\\
=&\int_{S^{n-1}}\bigg(\frac{\partial}{\partial t}\widetilde{V}_{q-1}\bigg)\rho^nd\xi+\int_{S^{n-1}}\widetilde{V}_{q-1}\bigg(\frac{\partial}{\partial t}\rho^n\bigg)d\xi\\
=&n^{-1}(q-1)\int_{S^{n-1}}\bigg(\int_{S^{n-1}}\rho^{q-2}\frac{\partial\rho}{\partial t}du\bigg)\rho^nd\xi+n\int_{S^{n-1}}\widetilde{V}_{q-1}\rho^{n-1}\bigg(\frac{\partial\rho}{\partial t}\bigg)d\xi\\
\leq&n^{-1}\rho^{-1}(q-1)Q(x_0,t)(\varepsilon_0h-\rho^2)\int_{S^{n-1}}\bigg(\int_{S^{n-1}}\rho^{q-2}du\bigg)\rho^nd\xi\\
&+n\rho^{-1}Q(x_0,t)(\varepsilon_0h-\rho^2)\int_{S^{n-1}}\widetilde{V}_{q-1}\rho^{n-1}d\xi,\end{align*}
and
\begin{align*}&-\frac{\partial}{\partial t}\bigg(\int_{S^{n-1}}hf/\varphi dx\bigg)=-\int_{S^{n-1}}\frac{\frac{\partial h}{\partial t}f}{\varphi}+\frac{hf\varphi^\prime\frac{\partial h}{\partial t}}{\varphi^2}dx\\
=&\int_{S^{n-1}}\bigg(\frac{f}{\varphi}+\frac{hf\varphi^\prime}{\varphi^2}\bigg)Q(h-\varepsilon_0)dx\leq\int_{S^{n-1}}\bigg(\frac{fh}{\varphi}+
\frac{h^2f\varphi^\prime}{\varphi^2}\bigg)dxQ(x_0,t).
\end{align*}
Thus,
\begin{align*}\frac{\partial}{\partial t}(\theta(t))\leq&\frac{\rho^{-1}Q(x_0,t)(\varepsilon_0h-\rho^2)\bigg((q-1)\int_{S^{n-1}}\widetilde{V}_{q-2}\rho^{n}d\xi
+n\int_{S^{n-1}}\widetilde{V}_{q-1}\rho^{n-1}d\xi\bigg)}{\int_{S^{n-1}}hf/\varphi dx}\\
&+\frac{\int_{S^{n-1}}\bigg(\frac{fh}{\varphi}+
\frac{h^2f\varphi^\prime}{\varphi^2}\bigg)dxQ(x_0,t)\bigg(\int_{S^{n-1}}
\widetilde{V}_{q-1}\rho^nd\xi\bigg)}{\bigg(\int_{S^{n-1}}hf/\varphi dx\bigg)^2}\\
=&Q(x_0,t)\bigg[\frac{\rho^{-1}(\varepsilon_0h-\rho^2)\bigg((q-1)\int_{S^{n-1}}\widetilde{V}_{q-2}\rho^{n}d\xi
+n\int_{S^{n-1}}\widetilde{V}_{q-1}\rho^{n-1}d\xi\bigg)}{\int_{S^{n-1}}hf/\varphi dx}\\
&+\frac{\int_{S^{n-1}}\bigg(\frac{fh}{\varphi}+
\frac{h^2f\varphi^\prime}{\varphi^2}\bigg)dx\bigg(\int_{S^{n-1}}
\widetilde{V}_{q-1}\rho^nd\xi\bigg)}{\bigg(\int_{S^{n-1}}hf/\varphi dx\bigg)^2}\bigg]\\
\leq&C_2Q.\end{align*}

We use (\ref{eq207}), (\ref{eq407}) and recall $b_{ij}=\nabla_{ij}h+h\delta_{ij}$ may give

\begin{align*}\frac{\partial (\det(\nabla^2h+hI))^{-1}}{\partial t}=&-(\det(\nabla^2h+hI))^{-2}
\frac{\partial(\det(\nabla^2h+hI))}{\partial b_{ij}}\frac{\partial(\nabla^2h+hI)}{\partial t}\\
=&-(\det(\nabla^2h+hI))^{-2}\frac{\partial(\det(\nabla^2h+hI))}{\partial b_{ij}}
(h_{tij+h_t\delta_{ij}})\\
\leq&(\det(\nabla^2h+hI))^{-2}\frac{\partial(\det(\nabla^2h+hI))}{\partial b_{ij}}
Q(b_{ij}-\varepsilon_0\delta_{ij})\\
\leq&\mathcal{K}Q((n-1)-\varepsilon_0(n-1)\mathcal{K}^{\frac{1}{n-1}}).
\end{align*}

Therefore, we have following conclusion at $x_0$,
\begin{align*}
\frac{\partial}{\partial t} Q\leq\frac{1}{h-\varepsilon_0}\bigg(CQ^2+f\theta h(\varphi\widetilde{V}_{q-1})^{-1}\mathcal{K}Q((n-1)-\varepsilon_0(n-1)\mathcal{K}^{\frac{1}{n-1}})\bigg).
\end{align*}
If $Q>>1$ ($>>$: far greater than),
\begin{align*}
\frac{1}{C_0}\mathcal{K}\leq Q\leq C_0\mathcal{K},
\end{align*}
which implies
\begin{align*}
\frac{\partial}{\partial t} Q\leq C_1Q^2(C_2-\varepsilon_0Q^{\frac{1}{n-1}})<0,
\end{align*}
since $C_1$ and $C_2$ depend on $\|f\|_{C^0(S^{n-1})}, \|\varphi\|_{C^1(I_{[0,T)})} ,\|\varphi\|_{C^2(I_{[0,T)})}$, $\|h\|_{C^1(S^{n-1}\times [0,T)}$ and $\|\theta\|_{C^0(S^{n-1}\times [0,T)}$. Hence, above ODE tells that
\begin{align*}
Q(x_0,t)\leq C,
\end{align*}
and for any $(x,t)$,
\begin{align*}
\mathcal{K}(x,t)=\frac{(h-\varepsilon_0)Q(x,t)+h}{f(x)h(\varphi\widetilde{V}_{q-1})^{-1}\theta}\leq
\frac{(h-\varepsilon_0)Q(x_0,t)+h}{f(x)h(\varphi\widetilde{V}_{q-1})^{-1}\theta}\leq C.
\end{align*}

Step 2: Prove $\kappa_i\geq\frac{1}{C}$.

We consider the auxiliary function as follows
\begin{align*}
\omega (x,t)=\log\lambda_{\max}(\{b_{ij}\})-A\log h+B|\nabla h|^2,
\end{align*}
where $A, B$ are positive constants which will be chosen later, and $\lambda_{\max}(\{b_{ij}\})$ denotes the maximal eigenvalue of $\{b_{ij}\}$; for convenience, we write $\{b^{ij}\}$ for $\{b_{ij}\}^{-1}$.

For every fixed $t\in[0,T)$, suppose $\max_{S^{n-1}}\omega(x,t)$ is attained at point $x_0\in S^{n-1}$. By a rotation of coordinates, we may assume
\begin{align*}
\{b_{ij}(x_0,t)\} \text{ is diagonal, and } \lambda_{\max}(\{b_{ij}\})(x_0,t)=b_{11}(x_0,t).
\end{align*}
Hence, in order to show $\kappa_i\geq\frac{1}{C}$, that is to prove $b_{11}\leq C.$ By means of the above assumption, we
transform $\omega(x,t)$ into the following form,
\begin{align*}
\widetilde{\omega}(x,t)=\log b_{11}-A\log h+B|\nabla h|^2.
\end{align*}
Utilizing again the above assumption, for any fixed $t \in [0,T)$, $\widetilde{\omega}(x,t)$ has a local
maximum at $x_0$, thus, we have at $x_0$
\begin{align}\label{eq408}
0=\nabla_i\widetilde{\omega}=&b^{11}\nabla_ib_{11}-A\frac{h_i}{h}+2B\sum h_kh_{ki}\\
\nonumber=&b^{11}(h_{i11}+h_1\delta_{i1})-A\frac{h_i}{h}+2Bh_ih_{ii},
\end{align}
and
\begin{align*}
	0\geq&\nabla_{ii}\widetilde{\omega}\\
	=&\nabla_ib^{11}(h_{i11}+h_1\delta_{i1})+b^{11}
	[\nabla_i(h_{i11}+h_1\delta_{i1})]-A\bigg(\frac{h_{ii}}{h}-\frac{h_i^2}{h^2}\bigg)
	+2B(\sum h_kh_{kii}+h^2_{ii})\\
	=&-(b_{11})^{-2}\nabla_ib_{11}(h_{i11}+h_1\delta_{i1})+b^{11}(\nabla_{ii}b_{11})-A\bigg(\frac{h_{ii}}{h}-\frac{h_i^2}{h^2}\bigg)
	+2B(\sum h_kh_{kii}+h^2_{ii})\\
	=&b^{11}\nabla_{ii}b_{11}-(b^{11})^2(\nabla_ib_{11})^2-A\bigg(\frac{h_{ii}}{h}-\frac{h_i^2}{h^2}\bigg)
	+2B(\sum h_kh_{kii}+h^2_{ii}).	
\end{align*}
At $x_0$, we also have
\begin{align*}
\frac{\partial}{\partial t}\widetilde{\omega}=&\frac{1}{b_{11}}\frac{\partial b_{11}}{\partial t}-A\frac{h_t}{h}+2B\sum h_kh_{kt}\\ =&b^{11}\frac{\partial}{\partial t}(h_{11}+h\delta_{11})-A\frac{h_t}{h}+2B\sum h_kh_{kt}\\
	=&b^{11}(h_{11t}+h_t)-A\frac{h_t}{h}+2B\sum h_kh_{kt}.
\end{align*}

From equation (\ref{eq303}) and (\ref{eq207}), we know
\begin{align}\label{eq409}
	\nonumber \log(h-h_t)=&\log(h+\theta(\varphi\widetilde{V}_{q-1})^{-1}\mathcal{K}hf-h)\\
	\nonumber=&\log\mathcal{K}+
	\log(\theta(\varphi\widetilde{V}_{q-1})^{-1}hf)\\
	=&-\log(\det(\nabla^2h+hI))+\log(\theta(\varphi\widetilde{V}_{q-1})^{-1}hf).
\end{align}
Let
\begin{align*}
	\phi(x,t)=\log(\theta(\varphi\widetilde{V}_{q-1})^{-1}hf).
\end{align*}
Differentiating (\ref{eq409}) once and twice, we respectively get
\begin{align*}\frac{h_k-h_{kt}}{h-h_t}=&-\sum b^{ij}\nabla_kb_{ij}+\nabla_k\phi\\
	=&-\sum b^{ii}(h_{kii}+h_i\delta_{ik})+\nabla_k\phi,
\end{align*}
and
\begin{align*}
	\frac{h_{11}-h_{11t}}{h-h_t}-\frac{(h_1-h_{1t})^2}{(h-h_t)^2}=&-\sum (b^{ii})^2(\nabla_{i}b_{ii})^2+b^{ii}\nabla_{ii}b_{ii}+\nabla_{11}\phi\\
	=&-\sum b^{ii}\nabla_{11}b_{ii}+\sum b^{ii}b^{jj}(\nabla_1b_{ij})^2+\nabla_{11}\phi.
\end{align*}
By the Ricci identity, we have
\begin{align*}
	\nabla_{11}b_{ii}=\nabla_{ii}b_{11}-b_{11}+b_{ii}.
\end{align*}
Thus, we can derive
\begin{align*}
\frac{\frac{\partial}{\partial t}\widetilde{\omega}}{h-h_t}=&b^{11}\bigg(\frac{h_{11t}+h_t}
	{h-h_t}\bigg)-A\frac{h_t}{h(h-h_t)}
	+\frac{2B\sum h_kh_{kt}}{h-h_t}\\
	=&b^{11}\bigg(\frac{h_{11t}-h_{11}}{h-h_t}+\frac{h_{11}+h-h+h_t}{h-h_t}\bigg)-A\frac{1}{h}
	\frac{h_t-h+h}{h-h_t}
	+\frac{2B\sum h_kh_{kt}}{h-h_t}\\
	=&b^{11}\bigg(-\frac{(h_1-h_{1t})^2}{(h-h_t)^2}+\sum b^{ii}
	\nabla_{11}b_{ii}-\sum b^{ii}b^{jj}(\nabla_1b_{ij})^2-\nabla_{11}\phi\bigg.\\
	&\bigg.+\frac{h_{11}+h-{(h-h_t)}}
	{h-h_t}\bigg) -\frac{A}{h}\bigg(\frac{-(h-h_t)+h}{h-h_t}\bigg)+\frac{2B\sum h_kh_{kt}}{h-h_t}\\
	=&b^{11}\bigg(-\frac{(h_1-h_{1t})^2}{(h-h_t)^2}+\sum b^{ii}
	\nabla_{11}b_{ii}-\sum b^{ii}b^{jj}(\nabla_1b_{ij})^2-\nabla_{11}\phi\bigg)\\
	&+b^{11}\bigg(\frac{h_{11}+h}{h-h_t}-1\bigg) +\frac{A}{h}-\frac{A}{h-h_t}+\frac{2B\sum h_kh_{kt}}{h-h_t}\\
	=&b^{11}\bigg(-\frac{(h_1-h_{1t})^2}{(h-h_t)^2}+\sum b^{ii}
	\nabla_{11}b_{ii}-\sum b^{ii}b^{jj}(\nabla_1b_{ij})^2-\nabla_{11}\phi\bigg)+\frac{1-A}{h-h_t}\\
	&-b^{11}+\frac{A}{h}+\frac{2B\sum h_kh_{kt}}{h-h_t}\\
    \leq&b^{11}\bigg(\sum b^{ii}(\nabla_{ii}b_{11}-b_{11}+b_{ii})-\sum b^{ii}b^{jj}(\nabla_1b_{ij})^2\bigg)
	-b^{11}\nabla_{11}\phi+\frac{1-A}{h-h_t}\\
	&+\frac{A}{h}+\frac{2B\sum h_kh_{kt}}{h-h_t}\\
	\leq&\sum b^{ii}\bigg[(b^{11})^2(\nabla_ib_{11})^2+A\bigg(\frac{h_{ii}}{h}-\frac{h_i^2}{h^2}
	\bigg)-2B(\sum h_kh_{kii}+h_{ii}^2)\bigg]\\
	&-b^{11}\sum b^{ii}b^{jj}(\nabla_1b_{ij})^2-b^{11}\nabla_{11}\phi+\frac{1-A}{h-h_t}+\frac{A}{h}+\frac{2B\sum h_kh_{kt}}{h-h_t}\\
	\leq&\sum b^{ii}\bigg[A\bigg(\frac{h_{ii}+h-h}{h}-\frac{h_i^2}{h^2}\bigg)\bigg]+2B
	\sum h_k\bigg(-\sum b^{ii}h_{kii}+\frac{h_{kt}}{h-h_t}\bigg)\\
	&-2B\sum b^{ii}(b_{ii}-h)^2-b^{11}\nabla_{11}\phi+\frac{1-A}{h-h_t}+\frac{A}{h}\\
	\leq&\sum b^{ii}\bigg[A\bigg(\frac{b_{ii}}{h}-1\bigg)\bigg]+2B\sum h_k\bigg(\frac{h_k}{h-h_t}
	+b^{kk}h_k-\nabla_k\phi\bigg)\\
	&-2B\sum b^{ii}(b_{ii}^2-2b_{ii}h)-b^{11}\nabla_{11}\phi+\frac{1-A}{h-h_t}+\frac{A}{h}\\
	\leq&-2B\sum h_k\nabla_k\phi-b^{11}\nabla_{11}\phi+(2B|\nabla h|-A)\sum b^{ii}-2B\sum b_{ii}\\
	&+4B(n-1)h+\frac{2B|\nabla h|^2+1-A}{h-h_t}+\frac{nA}{h}.
\end{align*}
Recall
\begin{align*}
\phi(x,t)=\log(\theta(\varphi\widetilde{V}_{q-1})^{-1}hf)=\log \theta-\log \varphi-\log\widetilde{V}_{q-1}+\log h+\log f,
\end{align*}
since $\theta$ is a constant factor, we have $\theta_k=0$.
Consequently, we may obtain following form by $\phi(x,t)$ and (\ref{eq408}),
\begin{align*}
&-2B\sum h_k\nabla_k\phi-b^{11}\nabla_{11}\phi\\
=&-2B\sum h_k\bigg(\frac{f_k}{f}+\frac{h_k}{h}-\frac{(\widetilde{V}_{q-1})_k}{\widetilde{V}_{q-1}}
-\frac{\varphi^{\prime}h_k}{\varphi}\bigg)-b^{11}\nabla_{11}\phi\\
=&-2B\sum h_k\bigg(\frac{f_k}{f}+\frac{h_k}{h}-\frac{(\widetilde{V}_{q-1})_k}{\widetilde{V}_{q-1}}
-\frac{\varphi^{\prime}h_k}{\varphi}\bigg)\\
&-b^{11}\bigg(\frac{ff_{11}-f_1^2}{f^2}+\frac{hh_{11}-h_1^2}{h^2}-\frac{\widetilde{V}_
{q-1}(\widetilde{V}_{q-1})_{11}-(\widetilde{V}_{q-1})^2_1}{(\widetilde{V}_{q-1})^2}-\frac{\varphi^{\prime \prime}h_1^2+\varphi^{\prime}h_{11}}{\varphi}
+\frac{(\varphi^{\prime}h_1)^2}{\varphi^2}\bigg)\\
\leq&C_1B+C_2b^{11}+2B\sum h_k\frac{(\widetilde{V}_{q-1})_k}{\widetilde{V}_{q-1}}+b^{11}\frac{h(b_{11}-h)}{h^2}\\
&+b^{11}\frac{\widetilde{V}_
{q-1}(\widetilde{V}_{q-1})_{11}-(\widetilde{V}_{q-1})^2_1}{(\widetilde{V}_{q-1})^2}+b^{11}\bigg(\frac{\varphi^{\prime \prime}h_1^2+\varphi^{\prime}h_{11}}{\varphi}
+\frac{(\varphi^{\prime}h_1)^2}{\varphi^2}\bigg),
\end{align*}
where $\varphi^{\prime\prime}=\frac{\partial^2\varphi(s)}{\partial s^2}$, for $q>2$, we have
\begin{align*}
2B\sum h_k\frac{(\widetilde{V}_{q-1})_k}{\widetilde{V}_{q-1}}=2B\sum h_k\frac{\frac{q-1}{n}\int_{S^{n-1}}\rho^{q-2}\rho_kdu}{\frac{1}{n}\int_{S^{n-1}}\rho^{q-1}du}\leq C_3B,
\end{align*}

\begin{align*}
b^{11}\bigg(\frac{\varphi^{\prime \prime}h_1^2+\varphi^{\prime}h_{11}}{\varphi}
+\frac{(\varphi^{\prime}h_1)^2}{\varphi^2}\bigg)=b^{11}\bigg(\frac{\varphi^{\prime \prime}h_1^2+\varphi^{\prime}(b_{11}-h)}{\varphi}
+\frac{(\varphi^{\prime}h_1)^2}{\varphi^2}\bigg)\leq C_4b^{11},
\end{align*}
and
\begin{align*}
&b^{11}\frac{\widetilde{V}_
{q-1}(\widetilde{V}_{q-1})_{11}-(\widetilde{V}_{q-1})^2_1}{(\widetilde{V}_{q-1})^2}\\
=&b^{11}\frac{\frac{1}{n}\int_{S^{n-1}}\rho^{q-1}du\bigg(\frac{(q-1)(q-2)}{n}\int_{S^{n-1}}
(\rho^{q-3}\rho_1^2\!+\!\rho^{q-2}\rho_{11})du\bigg)\!-\!\bigg(\frac{q-1}{n}\int_{S^{n-1}}\rho^{q-2}
\rho_1du\bigg)^2}{\bigg(\frac{1}{n}\int_{S^{n-1}}\rho^{q-1}du\bigg)^2},
\end{align*}
from $\rho=(h^2+|\nabla h|^2)^{\frac{1}{2}}$, we get $\rho_k=\rho^{-1}(hh_k+\Sigma h_kh_{kk})$, thus,
\begin{align*}
\rho_{11}=\frac{hh_{11}+h_1^2+\Sigma h_1h_{111}+\Sigma h_{11}^2}{\rho}-\frac{h_1^2b_{11}^2}{\rho^3}.
\end{align*}
This combined with (\ref{eq408}) implies
\begin{align*}
	\rho_{11}=\frac{h(b_{11}-h)+h_1^2+\Sigma h_1\bigg(A\frac{h_1}{h}-2Bh(b_{11}-h)-b^{11}(h_1\delta_{11})\bigg )b_{11}}{\rho}-\frac{h_1^2b_{11}^2}{\rho^3}.
\end{align*}
As a result, for $q>2$,  the above computations show that
\begin{align*}
	b^{11}\frac{\widetilde{V}_
		{q-1}(\widetilde{V}_{q-1})_{11}-(\widetilde{V}_{q-1})^2_1}{(\widetilde{V}_{q-1})^2}\leq C_5b^{11},
\end{align*}
It follows that
\begin{align*}
	\frac{\frac{\partial}{\partial t}\widetilde{\omega}}{h-h_t}\leq C_6B+C_7b^{11}+
	(2B|\nabla h|-A)\sum b^{ii}-2B\sum b_{ii}+4B(n-1)h+\frac{nA}{h}<0,
\end{align*}
provided $b_{11}>>1$ and if we choose $A>>B$. We obtain
\begin{align*}
	\widetilde{\omega}(x_0,t)\leq C,
\end{align*}
hence,
\begin{align*}
	\omega(x_0,t)=\widetilde{\omega}(x_0,t)\leq C.
\end{align*}
This tells us the principal radii are bounded from above, or equivalently $\kappa_i\geq\frac{1}{C}$.
\end{proof}

\section{\bf The convergence of the normalised flow}\label{sec5}

With the help of a priori estimates in the section \ref{sec4}, the long-time existence and asymptotic behaviour of the normalised flow (\ref{eq103}) (or (\ref{eq301})) are obtained, we also can complete proof of Theorem \ref{thm12}.
\begin{proof}[Proof of the Theorem \ref{thm12}] Since (\ref{eq303}) is parabolic, we can get its short time existence. Let $T$ be the maximal time such that $h(\cdot, t)$ is a positive, smooth and strictly convex solution
to (\ref{eq303}) for all $t\in[0,T)$. Lemma \ref{lem42}-\ref{lem44} enable us to apply Lemma \ref{lem45} to the equation (\ref{eq303}) and thus we can deduce a
uniformly upper and lower bounds for the biggest eigenvalue of $\{(h_{ij}+h\delta_{ij})(x,t)\}$. This implies
\begin{align*}
	C^{-1}I\leq (h_{ij}+h\delta_{ij})(x,t)\leq CI,\quad \forall (x,t)\in S^{n-1}\times [0,T),
\end{align*}
where $C>0$ independents on $t$. This shows that the equation (\ref{eq303}) is uniformly parabolic. Estimates for higher derivatives follows from the standard regularity theory of uniformly parabolic equations
Krylov \cite{KR}. Hence, we obtain the long time existence and regularity of solutions for the flow
(\ref{eq103}) (or (\ref{eq301})). Moreover, we obtain
\begin{align*}
	\|h\|_{C^{l,m}_{x,t}(S^{n-1}\times [0,T))}\leq C_{l,m},
\end{align*}
for some $C_{l,m}$ ($l, m$ are nonnegative integers pairs) independent of $t$, then $T=\infty$. Using parabolic comparison principle, we can attain the uniqueness of the smooth solution $h(\cdot,t)$ of equation (\ref{eq303}).

By the monotonicity of $\Phi$ in Lemma \ref{lem32}, there is a constant $C>0$ which is independent of $t$, such that
\begin{align}\label{eq501}
	|\Phi(X(\cdot,t))|\leq C,\ \ \ \ \ \forall t\in[0, \infty).
\end{align}
By the Lemma \ref{lem32}, we obtain
\begin{align}\label{eq502}\lim_{t\rightarrow\infty}\Phi(X(\cdot,t))-\Phi(X(\cdot,0))=
-\int_0^\infty|\frac{d}{dt}\Phi(X(\cdot,t))|dt.
\end{align}

By (\ref{eq501}), the left hand side of (\ref{eq502}) is bounded below by $-2C$, therefore, there is a
sequence $t_j\rightarrow\infty$ such that
\begin{align*}
	\frac{d}{dt}\Phi(X(\cdot,t_j))\rightarrow 0 \quad\text{as}\quad  t_j\rightarrow\infty,
\end{align*}
using Lemma \ref{lem32}, Lemma \ref{lem42} and Lemma \ref{lem43} again, above equation implies $h(\cdot,t)$ converges to a
positive and uniformly convex function $h_\infty\in C^\infty(S^{n-1})$ which satisfies (\ref{eq102})
with $c$ given by
\begin{align*}\frac{1}{c}=\lim_{t_j\rightarrow\infty}\theta(t_j).
\end{align*}
This completes the proof of Theorem \ref{thm12}.
\end{proof}
\vskip 1.0cm

{\bf Acknowledgement}
\vskip 0.6cm
The authors sincerely thank the referees for detailed reading and valuable comments to our paper.

\end{document}